\documentclass[journal]{IEEEtran}
\usepackage{amsmath}
\usepackage{amssymb}
\usepackage{amsthm}
\newtheorem{mydef}{Definition}
\newtheorem{thrm}{Theorem}
\newtheorem{remark}{Remark} 
\newtheorem{lem}{Lemma}
\newtheorem{cor}{Corollary}
\usepackage{graphicx} 
\usepackage{caption}
\usepackage{multicol}
\usepackage{multirow}
\usepackage[T1]{fontenc}
\usepackage{algorithm,algorithmic}
\theoremstyle{plain}
\newcommand{\norm}[1]{\left\lVert#1\right\rVert}
\usepackage{subfigure}
\ifCLASSINFOpdf
\else
  
\fi

\begin{document}
%

\title{ Stabilization of  Uncertain Discrete-time Linear System via Limited Information}
\author{Niladri Sekhar Tripathy,~\IEEEmembership{}
        I. N. Kar~\IEEEmembership{}
        and~Kolin Paul~\IEEEmembership{}
\thanks{Niladri Sekhar Tripathy and I. N. Kar are with the Dept. of  Electrical Engineering, Indian Institute of Technology Delhi, New Delhi, India e-mail: (niladri.tripathy@ee.iitd.ac.in,  ink@ee.iitd.ac.in).}
\thanks{Kolin Paul is with the Dept. of Computer science \& Engineering, Indian Institute of Technology Delhi, New Delhi, India e-mail: kolin@cse.iitd.ac.in. }
\thanks{}}
\markboth{}%
{Shell \MakeLowercase{\textit{et al.}}: Bare Demo of IEEEtran.cls for Journals}
\maketitle

\begin{abstract}
 This paper  proposes a procedure to control an uncertain discrete-time networked control system through a limited stabilizing input information. The system is primarily affected by the time-varying, norm bounded, mismatched parametric uncertainty. The input information is limited due to unreliability of communicating networks. An event-triggered based robust control strategy is  adopted  to capture  the network unreliability.  In event-triggered control the control input is computed and updated at the system end only when a pre-specified event condition is violated. The robust control input is derived to stabilize the  uncertain  system by solving an optimal control problem based on a virtual nominal dynamics and a modified cost-functional.  The designed robust control law with limited information ensures input-to-sate stability (ISS) of original system under presence of  mismatched uncertainty. Deriving the  event-triggering condition for discrete-time uncertain system and ensuring the stability  of such system  analytically  are the key contributions of this paper. A numerical example is given to prove the efficacy of the proposed event-based control algorithm over the conventional periodic one. 
    
\end{abstract}

\begin{IEEEkeywords}
Discrete-time event-triggered control, discrete-time robust control, mismatched uncertainty, optimal control, input to state stability.
\end{IEEEkeywords}

\IEEEpeerreviewmaketitle

\section{Introduction}\label{sec1}

Controlling uncertain cyber-physical systems (CPS) or networked control systems (NCS) subjected to limited information is a challenging task. Generally in CPS or NCS, multiple physical systems are interconnected and exchanged their local information through a  digital network. Due to shared nature of communicating network the continuous or periodic transmission of information causes a large bandwidth requirement.  Apart from bandwidth requirement, most  of the cyber-physical systems are powered by DC battery, so efficient use of power is essential. It is observed that  transmitting data over the communicating  network  has a proportional relation with the power consumption. This trade off motivates a large number of researchers to continue their research on NCS with minimal sensing and actuation \cite{nst2}-\cite{nstt34}, \cite{nst60}-\cite{nst44}. Recently an event-triggered based control technique is proposed by \cite{nst33}-\cite{nst35}, to reduce the information requirement for realizing a stabilizing control law.  In event-triggered control, sensing at system end and actuation at controller end happens only when a pre-specified event condition is violated. This event-condition mostly depends on system's current states or outputs.  The primary shortcoming of continuous-time event-triggered control is that it requires a continuous   monitoring of event condition. In \cite{nst60}-\cite{nst61}, Heemels et al. proposes an event-triggering technique where event-condition is monitored periodically.  To avoid continuous or periodic monitoring, self-triggered control technique is reported in \cite{nst37}-\cite{nst69p} where the next event occurring instant is computed analytically based on the state of previous instant.  Maximizing  inter-event time is the key aim of the event-triggered or self-triggered control in order to reduce the total transmission requirements.  A Girad \cite{nst35}, proposes a new event-triggering mechanism named as dynamic event-triggering to achieve  larger inter-event time with respect to previous approach \cite{nst33}. The above discussion says the efficacy of aperiodic sensing and actuation over the continuous or periodic one in the context of NCS \cite{nst31}.  \par In NCS, uncertainty  is mainly considered in communicating network in the form of time-delay,  data-packet loss in between transreceiving process \cite{nst47}, \cite{nst54}, \cite{nst67}. On the other hand the unmodeled dynamics, time-varying system parameters, external disturbances are the  primary sources of system uncertainty.  The main shortcoming of the classical event-triggered system lies in the fact that one must know the exact model of the plant apriori. A system with an uncertain  model is a more realistic scenario and has far greater significance. However, there are open problems of designing a control law and triggering conditions to deal with system uncertainties. To deal with parametric uncertainty, F. Lin et.al. proposes a continuous-time robust control technique where  control input is generated by solving an equivalent optimal control problem \cite{nst50}-\cite{nstf1}. The optimal control problem is formulated  based on the nominal or auxiliary dynamics by  minimizing a quadratic cost-functional which depends on the upper bound of uncertainty. The similar concepts is extended for nonlinear continuous system in \cite{nst51}, \cite{nst52}, where a non-quadratic cost-functional is considered.   However, this framework for discrete-time uncertain system is not reported.   Recently E. Garcia et.al. have proposed an event-triggered based discrete-time robust control technique for  NCS \cite{nst1}-\cite{nstt34}. To realize the robust control law their prior assumptions are that the physical system is affected by matched uncertainty (which is briefly discussed in section \ref{sec2}) and the  uncertainty is only in system's state matrix. But considering mismatched uncertainty  in both state and input matrices is more realistic control problem.   This is due to the fact that,  the existence of stabilizing control law can be guaranteed for matched uncertainty but difficult for  mismatched system.  Stabilization of mismatched uncertain system with communication constraint  is a challenging task.  This motivates us to formulate the present problem.  \\   
In this paper a novel discrete-time  robust control technique is proposed for NCS, where physical systems are inter-connected through an unreliable communication link. Due to unavailability of network the robust control law is designed using minimal state information. The designed control law acts on the physical system where the state model is affected  by mismatched parametric uncertainty.  For mathematical simplicity we are avoiding other uncertainties like external disturbances, noises.  The communication unreliability  is resolved by considering an event-triggered control technique, where control input is computed and updated only when an event-condition is satisfied. To derive the discrete-time robust control input a virtual nominal  system and a modified cost-functional is defined. Solution of the optimal control problem  helps to design the stabilizing control input for uncertain system. The input-to-state stability (ISS) technique is used to derive the event-triggering condition and as well as to ensure the stability of closed loop system. 
\subsection*{Highlights of contributions}
The contributions of this paper are summarized as follows:
\begin{description}
\item[(i)] Present paper proposes a robust control framework for discrete time linear system where state  matrix consists of mismatched  uncertainty.  The periodic robust control law is derived by formulating an equivalent optimal control problem. Optimal control problem is solved for an virtual system with a quadratic cost-functional which depends on the upper bound on uncertainty.  
\item[(ii)] The virtual nominal dynamics have two control inputs $u$ and $v$. The concept of  virtual  input $v$ is used to derive the existence of stabilizing control input $u$ to tackle mismatched uncertainty. The proposed robust control law ensures asymptotic convergence of uncertain closed loop system.  
 \item[(iii)] An event-triggered  robust control technique is proposed for discrete-time uncertain system, where controller  is not collocated with the system and  connected thorough a  communication network. The aim of this control law is to achieve robustness against parameter variation with event based communication and control.  The event condition is derived from the ISS based stability criteria.   
 \item[(iv)] It is shown that some of the existing results of matched system  is a special case of the proposed results \cite{nst1}-\cite{nstt34}. A comparative study is  carried out between  periodic control over the event-triggered control  on the numerical example.
 \end{description}
\subsection*{Notation \& Definitions}
  The notation $\|x\|$ is used to denote the Euclidean norm of a vector $x\in{\mathbb{R}^{n}}$. Here $\mathbb{R}^{n}$ denotes the $n$ dimensional Euclidean real space and $\mathbb{R}^{n\times {m}}$ is a set of all $(n\times{m})$  real matrices. $\mathbb{R}^{+}_{0}$ and $\mathbb{I}$ denote the all possible set of positive real numbers and non-negative integers. $X\leq{0}$, $X^{T}$ and $X^{-1}$ represent the negative definiteness, transpose and inverse of matrix $X$, respectively. Symbol $I$ represents an identity matrix with appropriate dimensions and time $t_{\infty}$ implies $+\infty$. Symbols $\lambda_{min}(P)$ and $\lambda_{max}(P)$ denote  the minimum and maximum eigenvalue of  symmetric matrix $P\in \mathbb{R}^{n\times n} $ respectively. 
  Through out this paper following definitions are used to derived the theoretical results \cite{nst69}.  
  \begin{mydef}\label{def23}
  A system 
  \begin{eqnarray}
  x(k+1)=Ax(k)+Bu(k)
  \end{eqnarray}
  is globally ISS if it satisfies 
  \begin{eqnarray}
  \|x(k)\|\leq\beta(\|x(0)\|,k)+\gamma\bigg( sup_{\tau\in[0,\infty)}\big\{\|u_{\tau}\|\big\}\bigg)
  \end{eqnarray}
 with each input $u(k)$ and each initial condition $x(0)$. The functions $\beta$ and $\gamma$ are $KL$ and $K_{\infty}$ functions respectively.
  \end{mydef}
  \begin{mydef}\label{def34}
  A discrete-time system $x(k+1)=f(x, u, d)$, whose origin is an equilibrium point if $f(0,x)=0$ $\forall \ k>0$. 
A  positive function $ V: \mathbb{R}^{n}\rightarrow \mathbb{R}$  is an ISS Lyapunov  function for that system  if there exist class $\emph{K}_{\infty}$ functions $ \alpha_{1}, \alpha_{2}, \alpha_{3}$ and a class  $\emph{KL}$ function  $ \gamma$ for all $x, d \in\mathbb{R}^{n}$ by satisfying the following conditions \cite{nst40}-\cite{nst41}.
\begin{eqnarray}\label{deft1}
\alpha_{1}(\|x(k)\|)\leq  V (x(k))\leq \alpha_{2}(\|x(k)\|)\\
V(x(k+1))-V(x(k))\leq -\alpha_{3}{(\|x(k)\|)}+\gamma{(\|d(k)\|)}\label{deft2}
\end{eqnarray}
\end{mydef}
\section{Robust Control Design}\label{sec2}
\emph{System Description:}
A discrete-time uncertain linear system is described by the state equation in the form 
\begin{eqnarray}\label{sys2}
x(k+1)=(A+\Delta A)x(k)+Bu(k)
\end{eqnarray}
where $x(k)\in \mathbb{R}^{n} $ is the state and $u(k)\in \mathbb{R}^{m}$ is the control input. The matrices $A\in \mathbb{R}^{n\times n}$ ,  $B\in\mathbb{R}^{n\times m}$ are the nominal, known constant matrices.  The unknown  matrices $\Delta A(p)\in \mathbb{R}^{n\times n} $ is used to represent the system uncertainty  due to bounded variation of   $p$. The uncertain parameter vector $p$  belongs to  a  predefined bounded set $\Omega$. Generally system uncertainties are classified as matched and  mismatched uncertainty. 
System (\ref{sys2}) suffers through  the matched uncertainty if the uncertain matrix $\Delta A(p)$  satisfy the following equality
\begin{eqnarray}\label{une3}
&&\Delta A(p)=B\phi(p) \  \\&&
\phi(p)^{T}\phi(p)\leq U_{matched}, \ \ \forall \  p\in \Omega
\end{eqnarray}
where $U_{matched}$ is the upper bound of uncertainty $\phi(p)^{T}\phi(p)$. In other words, $\Delta A(p)$ is in the range space of nominal input matrix $B$.  For mismatched case equality (\ref{une3}) does not holds.
For simplification,  uncertainty can be decomposed in matched and mismatched component such as
\begin{eqnarray}\label{une1}
\Delta A(p)= BB^{+}\Delta A+(I-BB^{+})\Delta A , \ \ \ \forall p\in \Omega
\end{eqnarray}
where $BB^{+}\Delta A$ is  matched and $(I-BB^{+})\Delta A$ is the mismatched one. 
The matrix $B^{+}$ denotes the pseudo inverse of  matrix $B$. 
 The perturbation $\Delta A$ is upper bounded by a known matrices $F$ and defined as
\begin{eqnarray}\label{scon2}
 \epsilon^{-1}\Delta A^{T}\Delta A\leq F
\end{eqnarray}
where, the scalar $\epsilon>0$ is a design parameter.  
 To stabilize (\ref{sys2}), it is essential  to formulate a  robust control problem as discussed below.   
\subsubsection*{Robust control problem}
Design a state feedback  controller law $u(k)=Kx(k)$ such that the uncertain closed loop system (\ref{sys2}) with the nominal dynamics
\begin{eqnarray}\label{syss1}
x(k+1)=Ax(k)+Bu(k)
\end{eqnarray} is asymptotically stable for all $p\in \Omega$.\\
In order  to stabilize (\ref{sys2}), the robust controller gain $K$ is designed through an optimal control approach. 
 \subsubsection*{Optimal control approach} The essential idea is to  design  an optimal control law for  a virtual nominal dynamics which minimizes a modified  cost-functional, $J$. The cost-functional is called modified as it  consists with the upper-bound of uncertainty.   An extra term $\alpha(I-BB^{+})v(k)$ is added with  (\ref{syss1}) to define virtual system (\ref{nomi1}). The derived optimal input for virtual system is also a robust input for original uncertain system.  The virtual  dynamics and  cost-functional for (\ref{sys2}) are given bellow:
 \begin{equation}\label{nomi1}
x(k+1)=Ax(k)+Bu(k)+\alpha(I-BB^{+})v(k)
\end{equation} 
\begin{eqnarray}\label{cos1}
J&=&\frac{1}{2}\bigg\{\sum_{k=0}^{\infty} x(k)^{T}Qx(k)+x(k)^{T}Fx(k)+\beta^{2} x(k)^{T}x(k)\nonumber \\ &&+u(k)^{T}R_{1}u(k)+v(k)^{T}R_{2}v(k)\bigg\}
\end{eqnarray}
where $Q\geq 0$,  $R_{1}>0$, $R_{2}>0$ and the scaler  $\alpha$ is a design parameter.
\begin{remark}
 Here $u=Kx$ is the stabilizing control input and $v(k)=Lx(k)$ is an virtual input.  The  input $v$ is called  virtual since it is not used directly to stabilize (\ref{sys2}). But it helps indirectly to design $K$.   The usefulness of L in the context of event-triggered  control  is discussed  in Section \ref{sec3}.  
\end{remark}
The  robust control law for (\ref{sys2}) is designed by minimizing (\ref{cos1}) for the virtual system (\ref{nomi1}). The results are stated in the form  a theorem. 
\begin{thrm}\label{thr12}
Suppose there exists a scalar $\epsilon>0$ and positive definite solution $P>0$ of equation (\ref{ri1}) 
\begin{eqnarray}\label{ri1}
&& A^{T}\big \{P^{-1}+BR_{1}^{-1}B^{T}+\alpha^{2}(I-BB^{+})R_{2}^{-1}\nonumber \\ &&(I-BB^{+})^{T}\big\}^{-1}A-P+Q+F+\beta^{2}I=0
\end{eqnarray}
with 
\begin{eqnarray}\label{scon1}
  (\epsilon^{-1}I-P)^{-1}>0,
   \end{eqnarray}
   \begin{eqnarray}\label{scon2}
 \epsilon^{-1}\Delta A^{T}\Delta A\leq F.
\end{eqnarray}
Moreover the controller gains $K$ and $L$ are computed  as
\begin{eqnarray}\label{g1}
K &=&-R_{1}^{-1}B^{T}\big\{P^{-1}+BR_{1}^{-1}B^{T}+\alpha^{2}(I-BB^{+})R_{2}^{-1}\nonumber \\ &&(I-BB^{+})^{T}\big\}^{-1}A \\ 
L&=&-\alpha R_{2}^{-1}(I-BB^{+})\big\{P^{-1}+B^{T}R_{1}^{-1}B+\alpha ^{2}(I-\nonumber \\ &&BB^{+})R_{2}^{-1}(I-BB^{+})^{T}\big\}^{-1}A \label{g2}
\end{eqnarray}
  These gains $K$ and $L$ are the  robust solution of (\ref{sys2}) $\forall \ p \in \Omega$, if it satisfies  following inequality:
\begin{eqnarray}\label{scon3}
\big(\beta^{2}I+L^{T}R_{2}L+K^{T}R_{1}K\big)- A_{c}^{T}\big(P^{-1}-\epsilon I\big)^{-1}A_{c}\geq 0
\end{eqnarray} 
where $A_{c}=A+BK$.
\end{thrm}
\begin{proof}
The proof of this theorem is given in Appendix \ref{A1}.
\end{proof}
The result  in Theorem \ref{thr12}  does not consider any communication constraint in realising the control law. So we formulate a  robust control problem  for an uncertain system with event-triggered control input. The block diagram of proposed robust control technique is shown in Figure \ref{fi:bld}. 
\begin{figure}[h]
\begin{center}
\includegraphics[width=8.5 cm,height=5 cm]{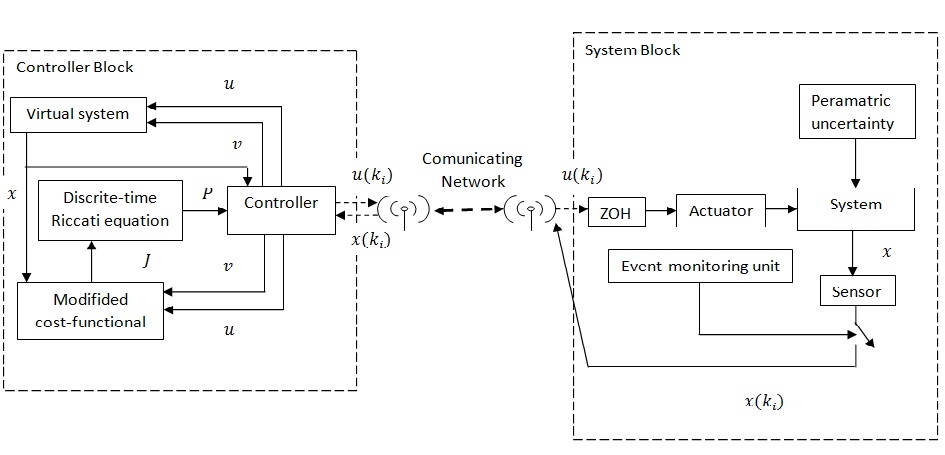}
\caption{Block diagram of proposed robust control law using  unrealiable communication channel.}\label{fi:bld}
\end{center}
\end{figure}
It has three primary parts namely,  system block ,  controller block and  a unreliable communicating network between system and controller. The states of system are periodically measured by the sensor which is collocated with the system. The sensor is connected with the controller through a communication network. An event-monitoring unit verifies a state-dependent event condition periodically and transmits state information to the controller only when the event-condition is satisfied.  The robust controller gain $K$ with eventual state information, $x(k_{i})$ which is received from uncertain system is used to generate the event-triggered  control law $u(k_{i})= Kx(k_{i})$ to stabilize (\ref{sys2}). Here $k_{i}$ denotes the latest event-triggering instant and  the control input is  updated at  aperiodic discrete-time instant $k_{0}, \ k_{1}, \ k_{2}$. A zero-order-hold (ZOH) is used to hold the last transmitted control input until the next input is transmitted. Here the actuator is collocated with the system and actuating control law is assumed to change instantly with the transmission of input. For simplicity, this paper does not consider any time-delay between sensing, computation and actuation instant. But in real-practice there must be some delay. This delay will effect the system analysis and as well as in event-condition.    \\
  A discrete-time linear uncertain system (\ref{sys2}) with event-triggered input
 \begin{eqnarray}\label{ine}
u(k_{i})=Kx(k_{i}), \ \ \ k\in\ [k_{i},\ k_{i+1} ) 
\end{eqnarray}  
is written as
\begin{eqnarray}\label{syse1}
x(k+1)=(A+\Delta A)x(k)+Bu(k_{i})
\end{eqnarray}
From  \cite{nst34}, this can be modelled as 
 \begin{eqnarray}
x(k+1)=(A+\Delta A +BK)x(k)+BKe(k)
 \end{eqnarray}
 The variable $e(k)$ is named as measurement error. It is used to represent the eventual state information $x(k_{i})$ in the form 
 \begin{eqnarray}\label{er}
 e(k)=x(k_{i})-x(k) \ \ \ k\in [k_{i}, \ k_{i+1})
 \end{eqnarray}
\textbf{Problem Statement:}
Design a feedback control law  to stabilize the uncertain discrete-time event-triggered system (\ref{syse1})  such that the closed loop system is ISS with respect to its measurement error $e(k)$.\\
\emph{Proposed solution:} \ This problem is solved in two different steps. Firstly, the controller is  designed by adopting the optimal control and secondly an event-triggering rule is derived to make (\ref{syse1})  ISS. The design procedure of controller gains, based on Theorems \ref{thr12}, is already discussed in Section \ref{sec2}. The event-triggering law is derived from the Definition \ref{def23} \& \ref{def34} assuming an ISS Lyapunov function $V(x)=x^{T}Px$. The design procedure of event-triggering condition is discussed elaborately in Section \ref{sec3}.    
\section{Main theoretical Results}\label{sec3}
This section discusses the design procedure of event-triggering condition and stability proof of (\ref{syse1})  under presence of bounded  parametric uncertainty. The results are stated in the form of a theorem.
\begin{thrm}\label{th1}
 Let $P>0$ be a solution of the  Riccati equation (\ref{ri1}) for a scalar $\epsilon>0$ and satisfy the following inequalities
\begin{eqnarray}\label{cons11}
Z=\epsilon^{-1}I+P(\epsilon^{-1}I-P)^{-1}P>0
\end{eqnarray}
\begin{eqnarray}\label{cons13}
 \Delta A^{T}Z\Delta A\leq F
\end{eqnarray}
\begin{eqnarray}\label{cons12}
\beta^{2}I+K^{T}R_{1}K+L^{T}R_{2}L-A_{c}^{T}ZA_{c}\geq 0
\end{eqnarray}
where controller gains $K$ and $L$ are computed by (\ref{g1}), (\ref{g2}). 
 The event-triggered control law (\ref{ine}) ensures the ISS of (\ref{syse1}) $\forall p\in \Omega$ if the input is updated through the following triggering-condition 
 \begin{eqnarray}
\norm{e(k)}^{2}\geq \mu_{1} \norm{x(k)}^{2}.
\end{eqnarray}
Moreover the design parameter $\mu_{1}$ is explicitly defined as 
\begin{eqnarray}\label{mu1}
\mu_{1}=\frac{\sigma \lambda_{min}(Q_{1})}{\norm{K^{T}B^{T}ZBK}} 
\end{eqnarray}
where scalar $ \sigma \  \in (0,1)$ and positive matrix $Q_{1}=(\beta^{2}I+K^{T}R_{1}K+L^{T}R_{2}L-A_{c}^{T}ZA_{c}) $.
\end{thrm}
To prove the above theorem some intermediate results are stated in the form of following lemmas. The proof of these lemmas are omitted due to limitation of pages. 
\begin{lem}\label{lem12}
Let $P$ be a positive definite solution of  (\ref{ri1}). Then there exist a scalar $\epsilon>0$ such that 
\begin{eqnarray}\label{cons1}
&& A_{c}^{T}P\Delta A+\Delta  A^{T}PA_{c}+\Delta A^{T}P\Delta A\leq  \nonumber \\&& A_{c}^{T}P(\epsilon^{-1}I-P)^{-1}PA_{c}+\epsilon^{-1}\Delta A^{T}\Delta A
\end{eqnarray} 
with
\begin{eqnarray}\label{cons2}
(\epsilon^{-1}I-P)^{-1}>0
\end{eqnarray}
\end{lem}
\begin{lem}\label{lem34}
Let  $P>0$ be a solution of  (\ref{ri1}) which satisfy the equation (\ref{cons11}). Using controller gain (\ref{g1}) and (\ref{g2}), the following inequality holds
\begin{eqnarray}\label{in67}
&& A_{c}^{T}Z A_{c} -A^{T}\big\{P^{-1}+BR_{1}^{-1}B^{T}+\alpha^{2}(I-BB^{+})\nonumber \\ &&  R_{2}^{-1}(I-BB^{+})^{T}\big\}^{-1}A   \leq  A_{c}^{T}(P^{-1}-\epsilon I)^{-1}) A_{c}\nonumber \\ &&  -(L^{T}R_{2}L+K^{T}R_{1}K)
\end{eqnarray}
\end{lem}
\begin{proof}[Proof of Theorem \ref{th1}]
Assuming $V(x)$ as a ISS Lyapunov function for (\ref{syse1}) and $A_{c}=(A+BK)$. The time difference of $V(x)$ $[\Delta V=V(x(k+1)-V(x(k))]$ along the direction of state trajectory   of (\ref{syse1}) \ $\forall \ p\in \Omega$ is simplified 
Using Lemma (\ref{lem12}).
\begin{eqnarray}\label{v2}
\Delta V &\leq & x^{T}[\epsilon^{-1} A_{c}^{T}A_{c}+A_{c}^{T}P(\epsilon^{-1}I-P)^{-1}PA_{c}-P]x\nonumber \\&&+x^{T}[\Delta A ^{T}P(\epsilon^{-1}I-P)^{-1}P\Delta A+\epsilon^{-1}\Delta A^{T}\Delta A]x\nonumber \\&& +e^{T}(\epsilon^{-1} K^{T}B^{T}BK+K^{T}B^{T}P(\epsilon^{-1}I-\nonumber \\&&P)^{-1}PBK)e
\end{eqnarray} 
For further simplification, the solution of Riccati equation (\ref{ri1}) $P$ is used in (\ref{v2}) to derive the following inequality
\begin{eqnarray}\label{v3}
\Delta V &\leq & x^{T}A_{c}^{T}ZA_{c}x-x^{T}A^{T}\{P(k+1)^{-1}+\nonumber \\&& BR_{1}^{-1}B^{T}+\alpha^{2}(I-BB^{+})R_{2}^{-1}(I-\nonumber \\&& BB^{+})^{T}\}^{-1}Ax-x^{T}(Q+\beta^{2}I)x-x^{T}(F\nonumber \\&&-\Delta A ^{T}Z\Delta A)x +e^{T}(K^{T}B^{T}ZBK)e
\end{eqnarray} 
where $Z$ is defined in (\ref{cons1}). 
Using Lemma \ref{lem34},   (\ref{v3}) is simplified as
\begin{eqnarray}\label{v4}
\Delta V&\leq & -x^{T}(\beta^{2}I+K^{T}R_{1}K+L^{T}R_{2}L-A_{c}^{T}ZA_{c})x\nonumber \\ && +e^{T}K^{T}B^{T}ZBKe
\end{eqnarray}
Applying Definition \ref{def34} in (\ref{v4}), the system (\ref{syse1})  remains ISS stable if  error $e(k)$ satisfies the following inequality
\begin{eqnarray}\label{evc1}
\norm{e}^{2}\geq \mu_{1}\norm{x}^{2}
\end{eqnarray}
where parameter 
\begin{eqnarray}\label{mu1}
\mu_{1}=\frac{\sigma \lambda_{min}(Q_{1})}{\norm{K^{T}B^{T}ZBK}}. 
\end{eqnarray}
The designed parameter $\sigma \in (0,1)$ and matrix $Q_{1}=(\beta^{2}I+K^{T}R_{1}K+L^{T}R_{2}L-A_{c}^{T}ZA_{c})>0$. The equation (\ref{evc1}) represents the event-triggering law for (\ref{syse1}).
\end{proof}
 The proposed robust control framework considers the general system uncertainty, which includes both matched and mismatched component.  Without mismatched part,  system (\ref{sys2})  reduces to matched  system (defined in (\ref{une3})), i.e. $\Delta A=B\phi(p) $ where  $\phi(p)=B^{+}\{A(p)-A(p_{0})\}$ .  Moreover, due to the absence of mismatched part,  the  virtual control input is not necessary in (\ref{nomi1}). As a special case  of Theorem \ref{thr12}, the Corollary \ref{cor1} is introduced for matched system. 
  \begin{cor}\label{cor1}
Suppose there exist a scalar $\epsilon>0$ and a positive definite solution $P$ of Riccati equation 
\begin{eqnarray}
 A^{T}\big \{P^{-1}+BR_{1}^{-1}B^{T})\big\}^{-1}A-P+Q+F+\beta^{2}I=0
\end{eqnarray}
with following conditions
\begin{eqnarray}
 \frac{2}{\epsilon}\phi^{T}B^{T}B\phi &\leq & F \\
 (\epsilon^{-1}I-P)^{-1}&>&0
\end{eqnarray}
Then $\forall p\in \Omega$, there exist a  controller gain  
\begin{eqnarray}
K &=&-R_{1}^{-1}B^{T}\big\{P^{-1}+BR_{1}^{-1}B^{T}\big\}^{-1}A 
\end{eqnarray}
which ensures the ISS of (\ref{sys2}) through the event-triggered control law 
\begin{eqnarray}\label{une82}
\|e\|^{2}\geq \mu_{2}\|x\|^{2}
\end{eqnarray}
 if the following sufficient condition holds
 \begin{eqnarray}\label{une81}
 (\beta^{2}I+K^{T}R_{1}K-\frac{2}{\epsilon}A_{c}^{T}A_{c})\geq 0.
\end{eqnarray}  
The parameter $\mu_{2}$ is derived as 
\begin{eqnarray}\label{une67}
\mu_{2}=\frac{\sigma\lambda_{min}(\tilde{Q} )}{\|K^{T}B^{T}(P^{-1}-\epsilon I)^{-1}BK\|}
\end{eqnarray}
 where scalar  $\sigma\in(0,1)$ and $\tilde{Q}=Q+F+\beta^{2}I$.
\end{cor}

\begin{remark}
The event-triggering law (\ref{evc1})  for (\ref{syse1})  is designed.  The designed parameter $\mu_{1}$  depends on the virtual gain $L$. Therefore $L$ has direct influence to design the robust stabilizing controller gain $K$ as well as on event-triggering law. 
The selection of design parameter $\beta$  and $\epsilon$ in (\ref{cos1})  has a greater significance on system stability due to the fact that the positive definiteness of (\ref{scon3}) depends on these parameters.
\end{remark}
\begin{remark}
Proposed framework can be extended in the presence of both system and input uncertainties and described by 
\begin{eqnarray}\label{sys1}
x(k+1)=\big[A+\Delta A(p)\big]x(k)+\big[B+\Delta B(p)\big]u(k)
\end{eqnarray} 
with a virtual dynamics (\ref{nomi1}) and  a cost-functional
\begin{eqnarray}\label{cosss2}
J&=&\frac{1}{2}\bigg\{\sum_{k=0}^{\infty} x(k)^{T}Qx(k)+x(k)^{T}Fx(k)+x(k)^{T}Hx(k)+\nonumber \\ && \beta^{2} x(k)^{T}x(k)+u(k)^{T}R_{1}u(k)+v(k)^{T}R_{2}v(k)\bigg\} 
\end{eqnarray}
 where $\beta>0$ and $H$ is a bound on $\Delta B$.
\end{remark}
\subsection*{Comparison with existing results}
This subsection compares the main contribution of this paper with the results reported in \cite{nst1}-\cite{nstt34}. To compare with \cite{nst1}, the mismatched part of uncertainty is neglected.   The Riccati equation mentioned in \cite{nst1} is  similar to   proposed Riccati equation (\ref{ri1}) in the presence of  matched uncertainty. The only difference is we have added a term $\beta^{2}x^{T}x$ in the cost-functional (\ref{ri1}).  Without this extra term the Riccati equation (\ref{ri1}) and controller gain (\ref{g1}) reduce to  similar form as mentioned in \cite{nst1}. So the results of (\ref{sys2}) can be recovered as a special case of  proposed results. 
\par
Moreover  in Theorem $1$ of \cite{nst1}-\cite{nstt34}, an event-triggering condition is stated for matched system which depends on uncertain matrix $\phi(p)$.  
Implementation of this  event-triggering condition is not realistic as matrix $\phi(p)$ is unknown. However in this paper, the proposed event triggering condition  discussed in Theorem \ref{th1} is independent of uncertainty $\phi(p)$. It directly depends  on the controller gain $K$.   Moreover the triggering condition for  matched system is also independent of uncertain matrix explicitly as seen from (\ref{une67}).   Thus, it is concluded that the proposed event-triggering condition is comparatively  more general and  easy to realize than \cite{nst1}-\cite{nstt34}.
\section{Numerical Examples and comparative studies}\label{sec4}
Consider an uncertain discrete-time linear system (\ref{syse1}) where $A+\Delta A=\begin{bmatrix}
p & 1+p \\ 1 & 0
\end{bmatrix}$, $B=\begin{bmatrix}
0 \\ 1
\end{bmatrix}$.
Here the system is mismatched in nature and the results of \cite{nst1}-\cite{nstt34} are not applicable.  To solve (\ref{ri1}), matrices $Q=I$, $R_{1}=I$, $R_{2}=I$, $F=\begin{bmatrix}
6.09 &6.09 \\ 6.09 & 6.09
\end{bmatrix}$ and $\alpha =10$ are selected. The design parameter $\epsilon=0.1$, $\beta=5$ , $\sigma=0.1$ and $\mu=0.29$  are used.  To design event-triggering condition $\sigma=0.1$ and $\mu=0.29$ are used. The uncertain parameter $p$ varies from $0$ to $0.8$. The simulation is carried-out in MATLAB for $20$ intervals with a sampling interval $T=1$ sec. The controller gains $K$ and $L$ are computed using  (\ref{g1}), (\ref{g2}) and their numerical values are  $K=[-0.9687 -0.0001]$ and $L=\begin{bmatrix}
-0.0006 & -0.1 \\ 0 & 0
\end{bmatrix}$.
\begin{figure}[h]
    \centering
    
    \subfigure[Convergence of uncertain states for event-triggered control for $p=0.8$.]
    {
        \includegraphics[width=3.5 in,height=1.5 in]{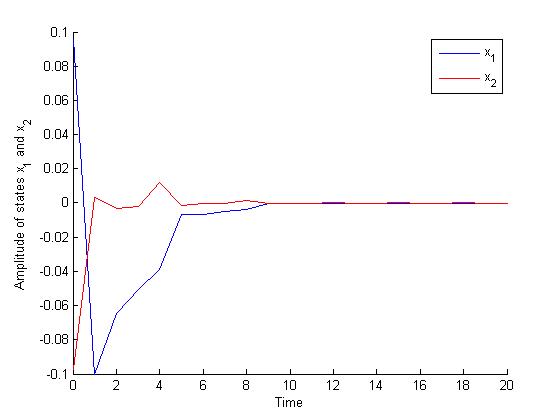}
        \label{fig:42}
    }
    \\
    \subfigure[Triggering instants at which control inputs are computed .]
    {
        \includegraphics[width=3 in, height=0.5 in]{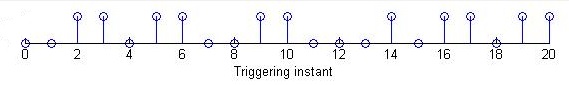}
       \label{fig:43}
    }
    
   \caption{Results of periodic and event-triggered control.}
\end{figure}

 Figure \ref{fig:42} shows the convergence of uncertain states for  event-triggered  control. For event triggered control the triggering instants are shown in Figure \ref{fig:43}. 
%
\section{Conclusion}\label{sec5}
A discrete-time periodic and aperiodic control of uncertain linear system is proposed in this paper. The control law is designed by formulating an optimal control problem for virtual nominal system with a modified cost-functional. An virtual input is defined to design the stabilizing controller gain along with the stability condition. The paper also proposes an event-triggered based control technique for NCS to achieve robustness. The event-condition and stability of uncertain system are derived using the ISS Lyapunov function. A new event-triggering law is derived which depends on the virtual controller gain in order to tackle mismatched uncertainty. A comparative study between existing and proposed results is also reported. \par
 A challenging future work to extend the proposed robust control framework for discrete-time nonlinear system with and without event-triggering input.
 This frame work can be formulated as a differential game problem where the control inputs $u$ and $v$ can be treated as minimizing and maximizing inputs. 
\appendices
\section{}\label{A1}
\begin{proof}[Proof of Theorem \ref{thr12}]
The proof has two parts. At first, we solve an optimal control problem to minimize (\ref{cos1}) for the nominal system (\ref{nomi1}).  For this purpose the optimal input $u$ and $v$ should minimize the Hamiltonian $H$, that means 
$\dfrac{\partial H}{\partial u}=0$ and $\dfrac{\partial H}{\partial v}=0$.
After applying discrete-time LQR methods, the Riccati equation (\ref{ri1}) and controller gains (\ref{g1}), (\ref{g2}) are achieved \cite{nst46}. To  prove the stability of uncertain system, let $V(x)=x(k)^{T}Px(k)$ be a Lyapunov function for (\ref{sys2}). Then applying (\ref{syse1}) the time difference of $V(x)$  along the $x$  is
\begin{eqnarray}\label{eq46}
\Delta V &=&  x(k)^{T}A_{c}^{T}PA_{c}x(k)+x(k)^{T}A_{c}^{T}P\Delta A x(k)\nonumber\\&&+x(k)^{T}\Delta APA_{c}x(k)+x(k)^{T}\Delta A^{T}P\Delta Ax(k)\nonumber\\&&-x(k)^{T}Px(k)
\end{eqnarray} 
where $A_{c}=(A+BK)$.
Using matrix inversion lemma, following  is achieved \cite{nst48}
\begin{eqnarray}\label{er2}
(P^{-1}-\epsilon I)^{-1}=P+P(\epsilon^{-1}I-P)^{-1}P
\end{eqnarray}
 Using (\ref{ri1}) and (\ref{er2}) in (\ref{eq46}), following is achieved
\begin{eqnarray}\label{grad1}
\Delta V &\leq& x(k)^{T}[A_{c}^{T}(P^{-1}-\epsilon I)^{-1}A_{c}-Q-\beta^{2}I- A^{T}\big \{P^{-1}\nonumber \\ && +BR_{1}^{-1}B^{T}+\alpha^{2}(I-BB^{+})R_{2}^{-1}(I-BB^{+})^{T}\big\}^{-1}\nonumber \\ &&A]x(k)-x(k)^{T}[F-\epsilon^{-1}\Delta A^{T}\Delta A]x(k)
\end{eqnarray}
Now applying  Lemmas \ref{lem12}, \ref{lem34}, the equation (\ref{grad1}) is simplified as
 \begin{eqnarray}\label{grad2}
\Delta V&\leq& -x(k)^{T} \big\{(\beta^{2}I+L^{T}R_{2}L+K^{T}R_{1}K)\nonumber\\&&- A_{c}^{T}(P^{-1}-\epsilon I)^{-1}) A_{c}\big\}x(k)
 \end{eqnarray}
 The inequality (\ref{grad2}) will be negative semi-definite if and only if equation (\ref{scon3}) is satisfied.
\end{proof}
\ifCLASSOPTIONcaptionsoff
  \newpage
\fi

\end{document}